\documentclass[11pt,reqno]{amsart}
\usepackage[utf8]{inputenc}
\usepackage[usenames,dvipsnames,svgnames,table]{xcolor}
\usepackage{amsthm, amssymb, amsmath,amsfonts, enumerate}
 \usepackage{fullpage}
\usepackage{verbatim}
\usepackage{graphicx, graphics}
\usepackage{algorithm}
\usepackage{multicol}
\usepackage{soul}
\usepackage{cancel,shuffle}
\usepackage{stackrel}
\usepackage{fdsymbol}
\usepackage{xargs}
\usepackage{pdflscape}
\usepackage{hyperref}
\usepackage{ulem}
\usepackage{float}

\usepackage{tikz}
\usetikzlibrary{arrows}
\tikzstyle{block}=[draw opacity=0.7,line width=1.4cm]

\usepackage[all,arc,curve,frame,color]{xy}
\usepackage{subfigure}
\usepackage{url}

\newtheorem{thm}{Theorem}[section]

\newtheorem{cor}[thm]{Corollary}
\newtheorem*{cor*}{Corollary}
\newtheorem{prop}[thm]{Proposition}

\newtheorem*{conjecture*}{Conjecture}

\theoremstyle{remark} 
\newtheorem*{question*}{Question}

\newtheorem{remark*}[thm]{Remark}

\theoremstyle{definition} 

\newtheorem*{define*}{Definition}

\numberwithin{equation}{section}  % number equations by section

\newcommand{\ZZ}{\mathbb{Z}} %the integers
\title[Diophantine equations and Valuation Trees]{Solving Quadratic and Cubic Diophantine 
Equations using 2-adic Valuation Trees}
\author{Maila Brucal-Hallare} 
 \address[M.~Brucal-Hallare]{Mathematics Department\\Norfolk State University\\700 Park Avenue\\ Norfolk, VA 23504, USA} 
 \email{mcbrucal-hallare@nsu.edu}
 
\author{Eva G. Goedhart} 
 \address[E.~Goedhart]{Department of Mathematics and Statistics\\Williams College\\
33 Stetson Court\\
Williamstown, MA 01267, USA} 
 \email{eva.goedhart@williams.edu}
 
  \author{Ryan Max Riley}
\address[R.~Riley]{Independent Scholar\\
 Williamstown, MA 01267, USA} 
 \email{riley@post.harvard.edu}

 \author{Vaishavi Sharma}
\address[V.~Sharma]{Department of Mathematics\\Tulane University\\6823 St Charles Ave\\New Orleans \\ LA 70118, USA} 
 \email{vsharma1@tulane.edu}
 
\author{Bianca Thompson}
\address[B.~Thompson]{Department of Mathematics\\Westminster College\\
1840 S 1300 E\\
 Salt Lake City, UT 84105, USA} 
 \email{bthompson@westminstercollege.edu}

\begin{document}
\maketitle

\begin{abstract}
    For fixed integers $D \geq 0$ and $c \geq 3$, we demonstrate how to use $2$-adic valuation trees of sequences to analyze Diophantine equations of the form $x^2+D=2^cy$ and $x^3+D=2^cy$, for $y$ odd. Further, we show for what values $D \in \mathbb{Z}^+$, the numbers $x^3+D$ will generate infinite valuation trees, which lead to infinite solutions to the above Diophantine equations. 
\end{abstract}

 \section{Introduction}

The generalized Lebesgue-Ramanujan-Nagell equations have been investigated using various methods \cite{lesoydan} including elementary techniques in classical number theory, Diophantine approximation methods, the Baker method, the Bilu-Hanrot-Voutier theorem, and the modular approach.

In this paper, we propose a visual approach to generalized Lebesgue-Ramanujan-Nagell equations. We present a straightforward way of analyzing the existence of solutions to families of Diophantine equations of the form
\begin{equation}\label{equationfamily}
x^2 + D = 2^c y,   
\end{equation}
or
\begin{equation}\label{equationfamilycubic}
x^3 + D = 2^c y,  
\end{equation}
where, in both cases, $c, D \in \mathbb{Z}^+.$  To solve these classic Diophantine problems, we employ the construction of valuation trees, which are binary decision trees that are driven by modular arithmetic  (see \cite{reuf2020} for more details on valuation trees). The valuation-tree approach makes it possible to visualize the relationships among the solutions. As we will demonstrate in this paper, the $p$-adic valuation tree makes solutions relatively straightforward to identify on these Diophantine equations, solutions to which have been elusive when analyzed by other means. 

In $2008$, Saradha and Srinivasan claimed that the number of solutions to Lebesgue-Nagell type equations of the form $x^2+D=\lambda y^n$, where $\lambda$, $D \in \mathbb{Z}^+$ are fixed and $x$, $y$, and $n$ are positive integers, is finite \cite[p.5]{saradha2008generalized}.  This might suggest that the number of positive integer solutions to the equation \eqref{equationfamily} is also finite. However, it depends on the choice of $D \in \mathbb{Z}^+$. 

Much work has been done on Lebesgue-Nagell type equations \cite{lesoydan}, but much less is known about the generalized Ramanujan-Nagell equation \eqref{equationfamily}.  The principal contributions include work by Bennett, Filaseta, and Trifonov who prove \cite[Theorem 1.1]{bennett2008yet} that if $x$, $y$ and $c$ are positive integers satisfying \eqref{equationfamily}, then either $x \in 1, 3, 5, 11, 181$ or $y > \sqrt{x}$.  This set of $x$ values is the one conjectured by Ramanujan and proven by Nagell to be the only integer solutions to $x^2 + 7 = 2^c$. By using valuation trees we'll be able to show a related result for \eqref{equationfamily} and prove the following theorem.

We fix the following notations:

\begin{tabular}{l l }
$\nu_p(n)=r$ & the $p$-adic valuation of $n$ where $n=p^rb$ and $p\nmid b.$\\
$\ell$ & is the level of the valuation tree, where $\ell \geq 0$\\
$D$ & the constant term in the Diophantine equation.\\
$c$ & the $2$-adic valuation of $x^2+D$ or $x^3 + D$\\
$(x,y,c)$ & a solution to the Diophantine equation $x^2+D = 2^cy$ or $x^3+D = 2^cy$

\end{tabular}

\begin{thm}\label{main theorem}
Let $D \in \mathbb{Z}^+$ be given. Consider the Diophantine equation $x^2+D=2^cy,$ with $2 \nmid y.$
\begin{enumerate}[a.]
\item If $D \neq 2^{2k}(8j+7), k,j \in \mathbb{Z}_{\geq 0}$,
then there exist infinitely many $c \in \mathbb{Z}$ such that $x^2 + D = 2^c y$ has no non-trivial integer solutions. 
\item If $D = 2^{2k}(8j+7), k,j \in \mathbb{Z}_{\geq 0},
$
then $x^2 + D = 2^c y$ has nontrivial solutions for all except a finite number of $c \in \mathbb{Z}$. Further, for $D=8j+7$ the set of values for $c$ does not include $c = 1$ or $2$.
\end{enumerate}
\end{thm}

The cubic equation (\ref{equationfamilycubic}) has a similar rich history of being studied. For $y=1$, Beukers \cite{beukers1981generalized} has proven that (\ref{equationfamilycubic}) has at most five solutions in $x\in\ZZ$. Alvarado et al \cite[Theorem 1.1]{sunit2020} analyzed the cubic Ramanujan-Nagell equation $x^3 + D = q^n$ for prime $q  > 3$ and $n,k > 0$, and where $D = 3^{k}$. Letting $q$ be a prime and $3 < q \leq 500$, they prove and list all integer solutions to the equation, and they claim that their method can also be used to find the integer solutions to the equation $x^3 + p^{k} = q^n$ where $p,q$ are different odd primes. Considering the Fermat type equation, $x^3 + y^3 = az^3$, for $a \in \mathbb{Z}$ with $a>2$ not divisible by the cube of any prime, Nagell \cite[p. 246-248]{nagellnumbertheory} proved that this family of Diophantine equations has either no solution or infinitely many solutions in relatively prime integers $x$, $y$, and $z$, with $z \neq 0.$ Nagell’s theorem can be used to prove that, for some values of $c$, $D$, and $y$, the related Diophantine equation \eqref{equationfamilycubic} has either no solution or infinite solutions. Although Nagell's theorem and method can be manipulated and used to find some solutions of the Diophantine equation \eqref{equationfamilycubic}, that approach is less comprehensive than solving the cubic Diophantine equation using $2$-adic valuation trees as we do in this paper. By using our valuation trees we prove the following theorem.

\begin{thm}\label{mainthmdeg3}
Let $D \in \mathbb{Z}^+$ be fixed. Consider the Diophantine equation $x^3+D=2^cy,$ with $y$ odd.
\begin{enumerate}[a.]
\item If $D \neq (2^{3k})(2j+1)$ for $k, j \in\ZZ_{\geq 0}$, then there are finitely many $c \in \mathbb{Z}$ for which the Diophantine equation has nontrivial solutions. Specifically, if
\begin{enumerate}
\item $D=2(2j+1),$ then $c=0$ or $1$ and $x$ is even or odd, respectively.
\item $D=2^2(2j+1)$, then $c=0$ or $2$ and $x$ is odd or even, respectively. 
\item $D=2^{3k+i},$ $i\in\{1,2\}$ and $k\in\mathbb{Z}^+$, then solutions are of the form $x\equiv 2^{\ell}\mod 2^{\ell+1}$ with corresponding $c=3\ell$ for $\ell<k+i.$ If $\ell=k+i$, then the form of our $x$ value is $x\equiv 0 \mod 2^{k+i}$ and $c=3k+i$. 
\end{enumerate}
\item If $D = 2^{3k}(2j+1)$ for $k, j \in\ZZ_{\geq 0}$, then $x^3 + D = 2^c y$ has nontrivial solutions for  all $c \geq 0$ except for finitely many values.
\end{enumerate}
\end{thm}

In Section~\ref{sec:tree}, we first prove that a special case of \eqref{equationfamily} for $c=0$ and all $c\geq 3$  
\begin{equation}\label{quad7}
x^2 + 7 = 2^cy,
\end{equation}
where $y$ is odd, has an infinite number of positive integer solutions. We use the valuation tree to determine the form of all of these solutions. We will use this section to demonstrate how the construction of valuation trees is done. In Sections ~\ref{sec:proof} and \ref{sec:cubic}, we employ valuation trees to prove Theorem \ref{main theorem} and Theorem \ref{mainthmdeg3}. In Section~\ref{sec:examples}, we conclude with a few examples of finding the nontrivial integer solutions for specific $D \in \mathbb{Z}^+.$ We discuss further directions of this approach in Section~\ref{sec:conclusion} and include examples in the Appendix.

\section{2-adic Valuation Tree for $x^2+7$}\label{sec:tree}
For an arbitrary integer $x$, we can create a binary tree to visualize the 2-adic valuation of $x^2+7$ for different values of $x$. In general, this value is difficult to determine. The value $\nu_2(x^2+7)$ will change depending on $x$ mod $2^\ell$ for each $\ell\in\ZZ^+$. A binary tree allows us to visualize the effect that the different classes of $x\mod 2^\ell$ have on $\nu_2(x^2+7)$, for $\ell>0$. For each $\ell\in\ZZ^+$, the horizontal level, $\ell$, of the tree represents all the possible equivalence classes of $x$ when we look at it $\mod 2^{\ell}$.

To see this explicitly, let $\ell=1$. Consider $x\equiv 0$ or $1\mod 2$. This is level $1$ on the tree in Figure~\ref{fig:plus7level1}, and $0,1$ are the labels on the edges or branches of the tree. For $x\equiv 0\mod 2$, write $x=2n$, an even number, for some $n\in\mathbb{Z}^+$. Then $\nu_2(x^2+7)=\nu_2((2n)^2+7)=0$, since the sum of an even and an odd is an odd number.

For $x$ odd, $x=2n+1$, where $n\in\mathbb{Z}^+$
\begin{align*}
  \nu_2((2n+1)^2+7) &= \nu_2(2^2(n^2+n+2))\\
  &= \nu_2(2^2)+\nu_2(n^2+n+2)\\
  &= 2+\nu_2(n^2+n+2)\\
  &\geq 2+1=3.
\end{align*}
The $1$ comes from the fact that (Case 1) if $n$ is even, then $n^2 + n + 2$ is also even so that there is at least $1$ factor of $2$; (Case 2) if $n$ is odd, then $n^2 + n + 2$ is even: say, $n = 2k+1$ implies that $n^2 + n + 2 = (4k^2 + 4k + 1) + (2k + 1) + 2 = 4k^2 + 6k + 4 = 2(2k^2 + 3k + 2)$, which has at least $1$ factor of $2$.

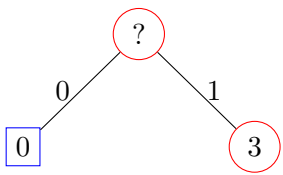
\begin{figure}[h]
\begin{tikzpicture}[sibling distance=8em,
]]
\tikzstyle{level 3}=[level distance=15mm,sibling distance=1.5cm]
\tikzstyle{level 4}=[level distance=15mm,sibling distance=1.2cm]
\tikzset{circle node/.style={shape=circle,draw=red,rounded corners,
    draw, align=center,
    top color=white, bottom color=white},square node/.style={shape=rectangle,draw=blue,
    draw, align=center,
    top color=white, bottom color=white}}
  \node[circle node] {$?$}
    child {node[square node] {0}edge from parent node[left]{$0$} }
      child { node[circle node] {3} 
      edge from parent node[right]{$1$}
        }
        ;
\end{tikzpicture}
\caption{Level 1 of the $2$-adic valuation tree for $x^2+7.$}
\label{fig:plus7level1}
\end{figure}

Depending on whether $n$ is even or odd, the $2$-adic valuation of our sequence may change. So we should subdivide our $x$ further.

Let us consider that $x$ is either $x=2^2n+1$ or $x=2^2n+3,$ that is, $x\equiv 1$ or $3\mod 4.$ We box the $1$ and $3$ below in the calculations to show what role they play in determining our valuations. 
We see that
\[(2^2n+\fbox{1})^2+7=2^4n^2+2^3n+1+7\]
and
\[(2^2n+\fbox{3})^2+7=2^4n^2+2^3(3n)+9+7\]
both depend on what $n$ will be to determine the $2$-adic valuation. So we get the tree in Figure~\ref{fig:plus7level2}. 

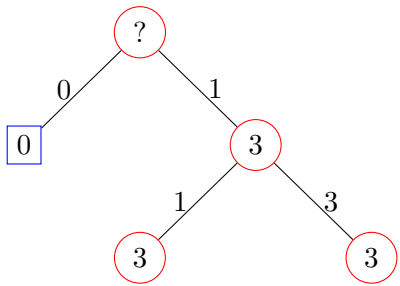
\begin{figure}[H]
\begin{tikzpicture}[sibling distance=8em,
]]
\tikzstyle{level 3}=[level distance=15mm,sibling distance=1.5cm]
\tikzstyle{level 4}=[level distance=15mm,sibling distance=1.2cm]
\tikzset{circle node/.style={shape=circle,draw=red,rounded corners,
    draw, align=center,
    top color=white, bottom color=white},square node/.style={shape=rectangle,draw=blue,
    draw, align=center,
    top color=white, bottom color=white}}
  \node[circle node] {$?$}
    child {node[square node] {0}edge from parent node[left]{$0$} }
    child { node[circle node] {3} 
      child { node[circle node]  {3}edge from parent node[left]{1}
         }
      child { node[circle node] {3}edge from parent node[right]{3} 
        %child{{}}
        } edge from parent node[right]{$1$}
        }
        ;
\end{tikzpicture}
\caption{The first two levels of the 2-adic valuation tree for $x^2+7.$}
\label{fig:plus7level2}
\end{figure}

So let's subdivide each of these values further with $x\equiv 1,3,5,7 \mod 8$
\[ x=2^3n+1,\ x=2^3n+5,  \ x=2^3n+3,\ x=2^3n+7.  \]

This order is preferred because ``1" and ``5" are on the left side of the tree, and ``3" and ``7" are on the right side of the tree, as in Figure~\ref{fig:plus7}.

Plugging these values into $x^2+7$ and simplifying yields

\begin{align*}
  2^6n+2^{3+1}\cdot\fbox{1}\cdot n+\fbox{1}^2+7\\
  2^6n+2^{3+1}\cdot\fbox{5}\cdot n+\fbox{5}^2+7\\
  2^6n+2^{3+1}\cdot\fbox{3}\cdot n+\fbox{3}^2+7\\
  2^6n+2^{3+1}\cdot\fbox{7}\cdot n+\fbox{7}^2+7\\
\end{align*}
Simplify each further and pull out the $2$'s to get
\begin{align}
    2^3(2^3n+2^{1}\cdot\fbox{1}\cdot n+1)\\\label{no2}
    2^4(2^2n+\fbox{5}\cdot n+2)\\
    2^4(2^2n+\fbox{3}\cdot n+1)\\
    2^3(2^3n+2^{1}\cdot\fbox{7}\cdot n+7)\label{no2too}
\end{align}
In equations \eqref{no2} and \eqref{no2too}, there are no more $2$'s that can be factored out, hence the $2$-adic valuations for \eqref{no2} and \eqref{no2too} is $3.$  But the other two equations should branch further as we have not yet determined their exact $2$-adic valuation. 

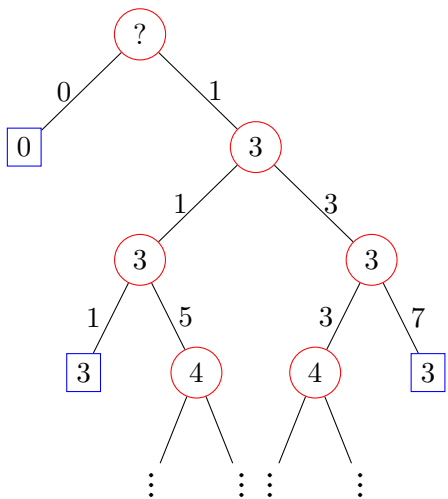
\begin{figure}[H]
\begin{tikzpicture}[sibling distance=8em,
]]
\tikzstyle{level 3}=[level distance=15mm,sibling distance=1.5cm]
\tikzstyle{level 4}=[level distance=15mm,sibling distance=1.2cm]
\tikzset{circle node/.style={shape=circle,draw=red,rounded corners,
    draw, align=center,
    top color=white, bottom color=white},square node/.style={shape=rectangle,draw=blue,
    draw, align=center,
    top color=white, bottom color=white}}
  \node[circle node] {$?$}
    child {node[square node] {0}edge from parent node[left]{$0$} }
    child { node[circle node] {3} 
      child { node[circle node]  {3}
        child { node[square node] {3} edge from parent node[left]{$1$}} 
        child { node[circle node]  {4} 
            child{node{$\vdots$}}
            child{node{$\vdots$}}
            edge from parent node[right]{$5$} 
            }edge from parent node[left]{$1$}
         }
      child { node[circle node] {3} 
        child{node[circle node]{4}
            child{node{$\vdots$}}
            child{node{$\vdots$}}edge from parent node[left]{$3$}} 
        child{node[square node]{3} edge from parent node[right]{$7$}}edge from parent node[right]{$3$}
        } edge from parent node[right]{$1$}
        }
        ;
\end{tikzpicture}
\caption{The first 3 levels of the 2-adic valuation tree for $x^2+7.$}
\label{fig:plus7}
\end{figure}

All of this information can be encoded in a binary decision tree. We begin at the initial node with $x^2+7$ where $x=n$, which is not yet categorized modulo $2^\ell,$ where $\ell$ is the level of our tree. Then the weight of each branch connected to this first node corresponds to the equivalence classes modulo $2$, because we're on the first level of our tree. The weights of the branches below each node correspond to the equivalence classes modulo $2^\ell$. Specifically they should be the two classes that are subsets of the equivalence class on the branch above them. For example in Figure \ref{fig:plus7}, we have branch $1$ and $5$ branching from the red circled 3 which is connected to the branch of weight $1$. This corresponds to being in the equivalence classes $1 \mod 2^3$ and $5\mod 2^3$, which are both the forms a number can take if they are also $1\mod 2^2.$ Further, we use blue squares to showcase when we know the exact valuation for a given classification of $x$, and red circles for when the valuation is still not known exactly, but we do know that it's greater than or equal to the value in the circle. We will always branch from the red circles and terminate on the blue squares. 

It has been shown \cite{reuf2020} that this tree is infinite, symmetrical, and has valuation range $\{0, 3, 4, 5, 6, ...\}$. Consider
\begin{equation}
\nu_2(x^2+7)=c,  \:\: \mbox{for an arbitrary  } c \geq3 
\end{equation}
We know by the 2-adic valuation tree for $c\geq 3$ there is always going to be a value for $x$ that makes the equation true. This means that 
\[x^2+7=2^cy\]
for some $y\in\ZZ$ where $2\nmid y.$ 
Thus, the 2-adic valuation range of $c$ corresponds to a Diophantine equation, and it will have non-trivial solutions. 
We now turn our attention to finding those solutions $(x,y,c)$ to the Diophantine equation $x^2+7=2^cy$.

 If we know a solution $(x,y,c)$ to $x^2+y=2^cy,$ $y$ odd, we can find the next solution recursively.

 \begin{prop}\label{prop:treerecursion}
The sequence of solution pairs $\{(x_{c}, y_{c})\}_{c \geq 3}$ of the Diophantine equation $x^2 + 7 = 2^{c}y$ corresponds to the minimum of the branch residues at level $\ell$ for $x_{c}$ and the non-terminating node behavior at level $\ell - 1$ for $y_{c}$, where $c = \ell + 1.$     
\end{prop}
 
 \begin{proof}
We know by ~\cite{reuf2020} that the tree for $x^2+7$ for every level $\ell \geq 3$ has four branches, each of which has a corresponding $2$-adic valuation node, and those branches are symmetric. Through the properties of modular arithemetic and the fact that the tree is symmetric, we know, among these four nodes, two of the nodes are terminating with 2-adic valuation equal to $\ell$ while the other two nodes are non-terminating nodes with 2-adic valuation at least $\ell + 1$. Further, we can describe exactly what the values of these branches are.

The two branches on the left trunk represent numbers of the form 
\[2^{\ell}n + x_{\ell},  \:\:\:\: 2^{\ell}n + x_{\ell} + 2^{\ell-1}\]
while the two branches on the right trunk represent numbers of the form 
\[2^{\ell}n -(2^{\ell-1} + x_{\ell}),  \:\:\:\: 2^{\ell}n + 2^{\ell}- x_{\ell}.\]

We will refer to the values $x_{\ell},\ x_{\ell} + 2^{\ell-1},\ -(2^{\ell-1}+x_{\ell}),$ and $-x_{\ell} $ mod $2^{\ell}$ as \textit{branch residues}.  Observe that for any given $\ell$, the four branch residues are determined by $x_{\ell}$.  The value of $x_{\ell}$ depends on the 2-adic valuation node at the previous level $\ell-1$.

The two branch residues on the left trunk, namely $x_{\ell}$ and $x_{\ell} + 2^{\ell - 1}$, emanate from a non-terminating node in the $\ell - 1$ level, that is, the 2-adic valuation node with value greater than or equal to $\ell$.  The branch residue $x_{\ell}$ is equal to the previous branch residue $x_{\ell-1}$.

The following property holds for the four branch residues: for the left trunk, we have

\[x_{\ell} < x_{\ell} + 2^{\ell -1} \mod 2^\ell,\]
while for the right trunk, we have

\[-(2^{\ell - 1} + x_{\ell}) <  - x_{\ell} \mod 2^{\ell}\]

Thus, in order to establish that $x_{\ell} = x_{c - 1}$ is the minimum branch residue at level $\ell$, we only need to check what the 2-adic valuation is at either the branch residue $x_{\ell} \mod 2^{\ell}$ or $-(2^{\ell - 1} + x_{\ell})\mod 2^{\ell}$.  Moreover, due to the symmetry of the tree, if the 2-adic valuation at the branch residue $x_{\ell}$ leads to a non-terminating node, then it follows that the branch residue at $-(2^{\ell - 1} + x_{\ell})\mod 2^{\ell}$ leads to a terminating node. As a result, the other non-terminating node will come from the branch having residue $ - x_{\ell}\mod 2^{\ell}$.  Observe that $x_{\ell} < 2^{\ell} - x_{\ell}$ for all $\ell > 1 $, because $x_{\ell} < 2^{\ell - 1}$.

Hence, without loss of generality, let us evaluate the 2-adic valuation along the branch $2^{\ell}k + x_{\ell}$; we get
$$
\begin{array}{lll}
\nu_{2}((2^{\ell}n + x_{\ell})^{2} + 7)  
&=& \nu_{2}(2^{2\ell}n^{2} + 2^{\ell+1}nx_{\ell} + x^{2}_{\ell} + 7)  \\
&=& \nu_{2}(2^{\ell}) + \nu_{2}(2^{\ell}n^{2} + 2^{\ell}nx_{\ell} + \frac{x^{2}_{\ell} + 7}{2^{\ell}})\\
&=& \ell + \nu_{2}(2^{\ell}n^{2} + 2^{\ell}nx_{\ell} + \frac{x^{2}_{\ell} + 7}{2^{\ell}})\\
&=& \ell + \nu_{2}(2^{\ell}n^{2} + 2^{\ell}nx_{\ell} + y_{\ell}),
\end{array}
$$
where $y_{\ell} = \dfrac{x_{\ell}^2 + 7}{2^{\ell}}$.  There are two cases to consider, either $y_{\ell} = y_{c - 1}$ is even or odd.  

If $y_{\ell} = y_{c- 1}$ is even, then 
$$
\nu_{2}(2^{\ell}n^{2} + 2^{\ell}nx_{\ell} + y_{\ell}) \geq 1
$$
so that $\nu_{2}((2^{\ell}n + x_{\ell})^{2} + 7) \geq \ell + 1$; this means that the branch $2^{\ell}n + x_{\ell}$ leads to a non-terminating node and hence, by symmetry, the branch $2^{\ell}n + 2^{\ell} - x_{\ell}$ leads to the other non-terminating node.  Since $x_{\ell} < 2^{\ell} - x_{\ell}$, we have established that the minimum value of the branch residues occur at $x_{\ell} = x_{c - 1}$ in this case.

If $y_{\ell} = y_{c- 1}$ is odd, then 
\[\nu_{2}(2^{\ell}n^{2} + 2^{\ell}nx_{\ell} + y_{\ell}) = 0\]
so that $\nu_{2}((2^{\ell}k + x_{\ell})^{2} + 7) = \ell$; this means that the branch residue $x_{\ell}$ leads to a terminating node and hence, the branch residue $2^{\ell - 1} -x_{\ell}$ leads to a non-terminating node.  The other non-terminating node comes from the branch having residue $x_{\ell} + 2^{\ell -1}$.  Since $2^{\ell - 1} - x_{\ell} < x + 2^{\ell - 1}$ for all $\ell > 1$, it follows that the minimum value of the branch residues occur at $2^{\ell - 1} - x_{\ell} = 2^{\ell - 2} - x_{c - 1}$ in this case.    
 \end{proof}
 
 \begin{cor}\label{thm:recursion}
If we know a solution of $x^2+7=2^{c-1}y$ is $(x_{c-1},y_{c-1},c-1)$ then a solution to $x^2+7=2^cy$ can be found by using the following recursion:
\[
x_{c} = \left\{
\begin{array}{ll}
x_{c-1}, & y_{c-1} \mbox{  is even} \\
2^{c - 2} - x_{c - 1}, &  y_{c - 1}\mbox{  is odd}
\end{array}
\right.
\]
and
\[
y_{c} = \dfrac{x_{c}^2+ 7}{2^c}.
\]

Note that $y_{c}$ as defined above is always a positive integer. 
 \end{cor}

\begin{thm}
The Diophantine equation $x^2+7=2^cy$ where $y$ is odd has solutions for all $c\geq 3$ and $c=0.$ Further, if we know solution $(x_{c-1},y_{c-1},c-1)$ for $x^2+7=2^{c-1}y$, we can find a solution for $x^2+7=2^cy$ by
\[x_{c} =2^{c - 2} - x_{c - 1}\]
because $y_{c-1}$ is assumed to be odd.

and
\[
y_{c} = \dfrac{x_{c}^2+ 7}{2^c}.
\]

The prior level of our tree corresponds with knowing a solution of our Diophantine equation $x^2+7=2^{c-1}y$. 

 \end{thm}
 
 \begin{proof}
We know by \cite{bhvlucas} that there are solutions to our equations if $c=1,2,3$, and those solutions are  $c=3,\ y=1$ or $c=2,\ y=2,$ or $c=1,\ y=4$ to this Diophantine equation when $x=1.$ If we further restrict $y$ to be odd, then we find that the $c=1$ and $c=2$ cases are actually a solution for $c=3$; $(x,y,c)=(1,1,3)$.

If $c=0$ then all solutions are of the form $(2k,(2k)^2+7,0)$ because we see from our tree that any even number plugged in for $x$ will always be odd and hence $c=0.$ To find $y$, we plug in our $x$ and $c$ value into our Diophantine equation and solve for $y.$ So
\[y=(2k)^2+7.\]
 
So now we know that for $y$ odd, $c\geq 3$  or $c=0$ have solutions. By \cite{reuf2020}, we know our valuation tree has range $\ZZ_{\geq 3}\cup \{0\}$, so we know we have a solution $x$ to $\nu_2(x^2+7)=c$ for all $c\geq 3$ and $c=0$. 
Further by Proposition~\ref{thm:recursion} we can recursively determine exactly the form $x$ should take for each $c$ value. See Table~\ref{tab:recursion} for some example solutions.
 \end{proof}

Notice that all of these solutions are in agreement with Bennett, Filaseta, and Trifonov \cite[Theorem 1.1]{bennett2008yet} because if $x$, $c$ and $y$ are positive integers satisfying \eqref{equationfamily}, then either $x$ is in the set of $1, 3, 5, 11, 181$ or $y > \sqrt{x}$. 

\section{Proof of Theorem~\ref{main theorem}}\label{sec:proof}

Medina, Moll, and Rowland \cite[Theorem 2.1]{victor2017} have proven that a polynomial with roots in $\ZZ_2,$ the $2$-adic integers, will form infinite valuation trees. They focus on the sequence $x^2+D$ in \cite[Lemma 3.8 and Theorem 4.5]{almodovaraclosed}, where they are able to describe the forms $D$ must take in order for the tree to be infinite. In particular if $D\equiv 7 \mod 8$, then the tree is infinite, or if $D\equiv 4\mod 8$ and $\nu_2(x_1^2+D)$ can be rewritten as $\nu_2(x_2^2+7)$, then the tree will be infinite. This is further expanded upon \cite[Theorem 1 part 3]{Olena2020} to general quadratic polynomials of the form $ax^2+bx+c.$ The authors specify that our $D$ should be of the form $2^{2k}(8j+7)$ if we wish to have infinite trees for some $k,j\in\ZZ_{\geq 0}$. Here we use these results to describe which Diophantine equations of the form $x^2+D=2^cy,$ $y$ odd, will have infinitely many nontrivial solutions for $c$.

\begin{proof}[Proof of Theorem~\ref{main theorem}]
If $D\neq 2^{2k}(8j+7),$ then we know the tree is bounded by \cite{almodovaraclosed} and \cite{Olena2020}. That means that the range of values $c$ can take is finite, and so there exists a $k$ such that for $c>k,$ $x^2+D=2^cy,$ $y$ odd has no nontrivial solutions. This is because no valuation branch exists with the value $c,$ hence $\nu_{2}(x^2+D)=c$ does not exist for $c>k.$ 

We know from Theorem 1 part 3b in ~\cite{Olena2020} that if $D=2^{2k}(8j+7)$, then that our tree will be infinite, that is, the range of $c$ will be infinite as well. Starting at $c>k,$ for some finite valuation of the tree $k$, there will be a valuation for every level of the tree. This means that our Diophantine equation $x^2+D=2^cy,$ $y$ odd, will have solutions for all $c$ in the range of valuations on the tree. 

Note that if $D\equiv 7\mod 8$ then $D=8j+7.$ We know that plugging an even number into $x^2+D$ will get us $\nu_2(x^2+D)=0$. Now, suppose $x=2n+1$, we then get  
\begin{align*}
\nu_2(x^2+D)&= \nu_2(x^2+8j+7)\\
&= \nu_2((2n+1)^2+8j+7)\\
&=\nu_2(2^2n^2+2^2n+8j+8)\\
&=\nu_2(2^3(\frac{n^2+n}{2}+j+1))\\
&=\nu_2(2^3)+\nu_2(\frac{n^2+n}{2}+j+1)) \geq 3\\
\end{align*}
Since $n^2+n$ is even for all $n$, we can divide out another $2$. So now we know from this work that if $D\equiv 7\mod 8$, $c$ will never be the values $1$ and $2$. 
Further, it is shown in \cite[Theorem 4.4]{almodovaraclosed} that if there is valuation $c$ on a terminating node of tree, then there will be valuation $c+1$ on the next level as a terminating node. So the range of $c$ values in the solution of $x^2+D=2^cy,$ $y$ odd and $D\equiv 7\mod 8$ is at most $\{0,3,4,5,\hdots\}.$ 
Using a similar proof for $D=2^{2k}(8j+7),$ $k\geq 1$ we can show that $c\neq 1.$
\end{proof}

\section{Cubic Diophantine Equation and finding infinite trees}\label{sec:cubic}

In order to understand the Diophantine results for $x^3+D=2^cy$, $y$ odd, we again put together results on when the $2$-adic valuation tree are finite and infinite. Based on reliable communications, these next two propositions have been proven for general prime $p$ by Victor Moll, Vaishavi Sharma, and Diego Villamizar. They use cubic reciprocity as well as exploring the roots of the polynomials in $\ZZ_p$. However, a preprint containing their results is not available as of this writing. Thus, here we've chosen to include a proof of the special case when $p=2$. 
\begin{prop}\label{prop:deg3finite}
%(Conjecture) 
If $D = 2^{3k}(2j+1)$ for $k, j \geq 0$ then the valuation tree of $x^3 + D$ 
is infinite. 
\end{prop}
\begin{proof}
When $k = 0$, this is easy to see as $D=2j+1$ and every integer is a cubic modulo $2$. So when D is odd, by Hensel's lemma we can see that the cube root of D is a root of the polynomial in $\ZZ_{2}$ and therefore we have an infinite branch~\cite[Theorem 2.1]{victor2017}.   
 When $k>0$ we see that the cube root of D is $2^k(2j+1)^{1/3}$ that is, $8^k$ is a perfect cube and $2j+1$ is an odd integer, so by Hensel's lemma, we have that the cube root of D is a root of the polynomial in $\ZZ_{2}$ and therefore we have an infinite branch. 
\end{proof}

Here we work out the cases of $D$ that would result in finite branches. 

\begin{prop}\label{prop:finitedeg3}
Let $D$ be a positive integer. 
If $D \not\in \{2^{3k}(2j+1): k, j \geq 0\}$ then
\begin{enumerate}
    \item 
If $D = 2(2j+1)$ then 

$$
\nu_{2}(x^3 + D)=\left\{
\begin{array}{cc}
1, & x \equiv 0 \mod 2\\
0, & x \equiv 1 \mod 2. 
\end{array}
\right.
$$

\item If $D= 2^2(2j+1)$ then
$$
\nu_{2}(x^3 + D)=\left\{
\begin{array}{cc}
2, & x \equiv 0 \mod 2\\
0, & x \equiv 1 \mod 2. 
\end{array}
\right.
$$

\item If $D = 2^{3k + i} $ for $k > 0$ and $i=1\text{ or } 2$ then
$$
\nu_{2}(x^3 + D) = \begin{cases}
0 & \text{ if }x \text{ is odd}\\
3 &  \text{ if }x\equiv 2\mod 4 \\
3\ell  & \text{ if }x \equiv 2^\ell\mod 2^{\ell+1}\\
\vdots \\
3k & \text{ if }x\equiv 2^k\mod 2^{k+1} \\
3k+i & \text{ if }x\equiv 0\mod 2^{k+1} 
\end{cases}
$$
where $\ell = 2, ..., k$.
\end{enumerate}
Hence, the valuation tree is finite. 
\end{prop}

\begin{proof}
\textit{(1).}  Suppose $D = 2(2j + 1)$ for some $j \geq 0$.  If $x = 2n$ then
$$
\nu_{2}((2n)^3 + 4j + 2) = \nu_{2}(2) + \nu_{2}(4n^3 + 2j + 1) = 1.
$$
If $x = 2n + 1$ then
$$
\nu_{2}((2n + 1)^3 + 4j + 2) = 0, $$
since $(2n + 1)^3 + 4j + 2$ is odd for all $j$.

\textit{(2).} Suppose $D = 2^2(2j + 1)$  for some $j \geq 0$.  If $x = 2n$ then
$$
\nu_{2}((2n)^3 + 8j + 4) = \nu_{2}(2^2) + \nu_{2}(2n^3 + 2j + 1) = 2.
$$
If $x = 2n + 1$ then
$$
\nu_{2}((2n + 1)^3 + 8j + 4) = 0, $$
since $(2n + 1)^3 + 8j + 4$ is odd for all $n$.

\textit{(3)} Suppose $D = 2^{3k+i}$ for $k > 0$ and $i\in\{1,2\}$ then if $x = 2n+1$ we have
$$
\nu_{2}((2n+1)^3 + 2^{3k+i}) = 0,
$$
while if $x = 2n$ then
$$
\nu_{2}((2n)^3 + 2^{3k+i}) = \nu_{2}(2^3) + \nu_{2}(n^3 + 2^{(3k+i)-3}) \geq 3.
$$
Since this valuation is not constant, we have to look at two cases, $i = 1$ and $i = 2$.

Suppose $i = 1$.  We will see that this case results in finite valuation trees that have different structures. The proof for $i = 2$ is analogous.

Diagramming the valuation tree, the first level of the tree in this case has two branches yielding a terminating valuation node value 0 (with residue 1) and a non-terminating node value 3 (with residue 0).  Moreover, the exponent $(3k+1)-3 = 3k-2$ may only take values from the set $\{1, 4, 7, ...\} = \{3j+1\}_{j=0}^{\infty}$ since $k$ is a natural number. 

To create the next level of the tree, consider $x = 2^2n$ and $x=2^2n + 2$:
\begin{equation}\label{3i-2}
\nu_{2}((2^2n)^3+2^{3k+1})=\nu_{2}(2^{3})+\nu_{2}(2^3n^3 + 2^{3k-2})\geq 3,
\end{equation}
$$
\nu_{2}((2^2n + 2)^3 +2^{3k+1})=\nu_{2}(2^{3})+\nu_{2}(2^3n^3 + 3\cdot 2^2 n^2 + 3 \cdot 2 n + 1 + 2^{3k-2})=3,
$$
since $3k-2>0$.  

Observe that the second equation in (\ref{3i-2}) gives a constant node value.  Let us now look at the first equation in (\ref{3i-2}).
If $3k-2 = 1$ then (\ref{3i-2}) becomes 
$$
\nu_{2}((2^2n)^3+2^{3k+1})=\nu_{2}(2^{3})+\nu_{2}(2^3n^3 + 2^1)= 4,
$$
and in this case, the valuation tree terminates:
 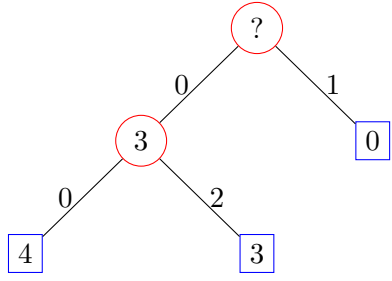
\begin{figure}[H]
\begin{tikzpicture}[sibling distance=8em,
]]
\tikzstyle{level 3}=[level distance=15mm,sibling distance=1.5cm]
\tikzstyle{level 4}=[level distance=15mm,sibling distance=1.2cm]
\tikzset{circle node/.style={shape=circle,draw=red,rounded corners,
    draw, align=center,
    top color=white, bottom color=white},square node/.style={shape=rectangle,draw=blue,
    draw, align=center,
    top color=white, bottom color=white}}
  \node[circle node] {$?$}
     child {node[circle node] {$3$}
        child { node[square node]  {$4$}  edge from parent node[left]{$0$} }
        child { node[square node]  {3} edge from parent node[right]{$2$} }
     edge from parent node[left]{$0$} 
     }
     child { node[square node] {$0$}
      edge from parent node[right]{$1$} 
    } 
        ;
\end{tikzpicture}
\caption{The finite valuation tree for $x^3+2^4.$}
\label{fig:cubeplus2}
\end{figure}

We see here that when $3k - 2=1$ then the 2-adic valuation tree terminates at the second level (with second-level branch residues $2^2n, 2^2 n + 2$) and having exact valuation nodes $0, 3, 4$.

If $3k-2 = 4$ then the first equation of (\ref{3i-2}) becomes
$$
\nu_{2}((2^2n)^3+2^{3k+1})= \nu_{2}(2^{3})+\nu_{2}(2^3n^3 + 2^{4}) = \nu_{2}(2^{3})+\nu_{2}(2^3) + \nu_{2}(n^3 + 2)\geq 6,
$$
which then requires that we move on to the next level of the tree.

Consider $x=2^3n$ and $x=2^3n + 4$:
\begin{equation}\label{3i-5}
\nu_{2}((2^3n)^3+2^{3k+1})=\nu_{2}(2^{6})+\nu_{2}(2^3n^3 + 2^{3k-5})\geq 6,
\end{equation}
$$
\nu_{2}((2^3n + 4)^3 +2^{3k+1})=\nu_{2}(2^{6})+\nu_{2}(2^3n^3 + 3\cdot 2^2 n^2 + 3 \cdot 2  + 1 + 2^{3k-5})=6,
$$
since $3k-5>0$. 
Observe that the second equation in (\ref{3i-5}) gives a constant node value. Now we look at the first equation in (\ref{3i-5}).

If $3k-5=1$ then the first equation of (\ref{3i-5}) becomes
$$
\nu_{2}((2^3n)^3+2^{3k+1})=\nu_{2}(2^{6})+\nu_{2}(2^3n^3 + 2^{1}) = 7,
$$
and in this case, the valuation tree terminates as shown in Figure~\ref{fig:cubeplus4}.

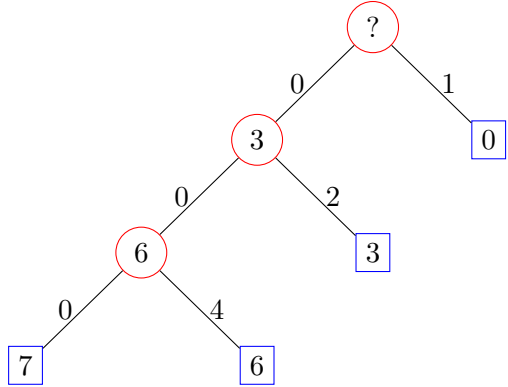
\begin{figure}[H]
\begin{tikzpicture}[sibling distance=8em,
]]
\tikzstyle{level 4}=[level distance=15mm,sibling distance=1.2cm]
\tikzset{circle node/.style={shape=circle,draw=red,rounded corners,
    draw, align=center,
    top color=white, bottom color=white},square node/.style={shape=rectangle,draw=blue,
    draw, align=center,
    top color=white, bottom color=white}}
  \node[circle node] {$?$}
     child {node[circle node] {$3$}
        child { node[circle node]  {$6$}  
        child{node[square node] {$7$}edge from parent node[left]{$0$} }
        child{node[square node] {$6$}edge from parent node[right]{$4$} } edge from parent node[left]{$0$} }
        child { node[square node]  {3} edge from parent node[right]{$2$} }
     edge from parent node[left]{$0$} 
     }
     child { node[square node] {$0$}
      edge from parent node[right]{$1$} 
    } 
        ;
\end{tikzpicture}
\caption{The finite valuation tree for $x^3+ 2^7.$}
\label{fig:cubeplus4}
\end{figure}

We see here that when $3k-5=1$ then the 2-adic valuation tree terminates at the third level (with third-level branch residues $2^3n, 2^3 n + 2^2$) and has exact valuation nodes $0, 3, 6, 7.$

The conclusion follows by induction on $k$.  This ends the proof for the case $i = 1$. The proof for $i = 2$ is analogous.

 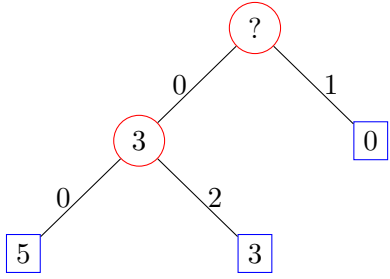
\begin{figure}[H]
 \begin{tikzpicture}[sibling distance=8em,
 ]]
 \tikzstyle{level 3}=[level distance=15mm,sibling distance=1.5cm]
 \tikzstyle{level 4}=[level distance=15mm,sibling distance=1.2cm]
 \tikzset{circle node/.style={shape=circle,draw=red,rounded corners,
     draw, align=center,
     top color=white, bottom color=white},square node/.style={shape=rectangle,draw=blue,
     draw, align=center,
     top color=white, bottom color=white}}
   \node[circle node] {$?$}
      child {node[circle node] {$3$}
         child { node[square node]  {$5$}  edge from parent node[left]{$0$} }
         child { node[square node]  {3} edge from parent node[right]{$2$} }
      edge from parent node[left]{$0$} 
      }
      child { node[square node] {$0$}
       edge from parent node[right]{$1$} 
     } 
         ;
 \end{tikzpicture}
 \caption{The finite valuation tree for $x^3+ 2^5$.}
 \label{fig:cubeplusi}
\end{figure}

\end{proof}

Now, we apply Propositions ~\ref{prop:deg3finite} and~\ref{prop:finitedeg3} to prove Theorem \ref{mainthmdeg3}.

\begin{proof}[Proof of Theorem ~\ref{mainthmdeg3}]
 Using the same techniques as our proof of Theorem~\ref{main theorem}  where we relate the valuations on the tree to our Diophantine equation we can show that for $D=2^{3k}(2j+1)$, $x^3+D=2^cy,$ $y$ odd, has a solution for all $c\geq 0$ except for finitely many values. This is because Proposition~\ref{prop:deg3finite} says the valuation trees are infinite. 

If $D\neq 2^{3k}(2j+1)$ we can see from Proposition~\ref{prop:finitedeg3} that the range of values of $c$ depends on $D = 2(2j+1)$, $D=2^2(2j+1)$ or $D=2^{3k+i},$ $i\in\{1,2\}$ and $k\in\mathbb{Z}^+$.  We know that there are finitely many $c$'s where $x^3+D=2^cy,$ $y$ odd, will have solutions. 

Note that we can do similar calculations like we did in the proof of Theorem~\ref{main theorem} to show that for $D=8^k(2j+1)$, $k\geq 1$ that $\nu_2(x^3+D)\geq 3.$ So our potential $c$ valuations in this case are at most $\{0,3,4,\hdots\}$. And if $k=0$ then we could at most have solutions for all $c\in\ZZ_{\geq 0}.$
\end{proof}

\section{Examples of using Valuation Trees to solve $x^2+D=2^cy$ and $x^3+D=2^cy$ for specific $D$}\label{sec:examples}
In \cite{reuf2020}, it was shown the exact forms of the valuation trees for $D=1,\ 2,\ 3$ and $4$. Using these trees we can find all the nontrivial solutions to the Diophantine equation $x^2+D=2^cy$ for $D=1,\ 2,\ 3$ and $4.$ Here we show what solutions for the quadratic Diophantine equation would be for $D=1,\ 3,$ and $4.$
\begin{thm}
The equation $x^2 + 1 = 2^c y$ has no solutions when $c \neq 0, 1$.  If $c=0$, $y$ is odd whenever $x$ is even.  If $c = 1$, $y$ is even whenever $x$ is odd. 

The solutions to $x^2+1=2^cy,$ where $2\nmid y$ are of the form $x=0\mod 2$ and $x=1\mod 2$ and $c=0$ or $1$, respectively. In this case, all solution $(x,y,c)$ are $(0\mod 2,x^2+1 ,0)$ and $(1\mod 2,\frac{x^2+1}{2},1).$
\end{thm}
 \begin{figure}[H]
\begin{tikzpicture}[sibling distance=8em,
]]
\tikzset{circle node/.style={shape=circle,draw=red,rounded corners,
    draw, align=center,
    top color=white, bottom color=white},square node/.style={shape=rectangle,draw=blue,
    draw, align=center,
    top color=white, bottom color=white}}
  \node[circle node] {$?$}
     child {node[square node] {0}edge from parent node[left]{$0$} }
     child { node[square node] {$1$} 
          edge from parent node[right]{$1$}
        }
        ;
\end{tikzpicture}
\caption{The finite valuation tree for $x^2+1$}
\label{fig:plus1}
\end{figure}
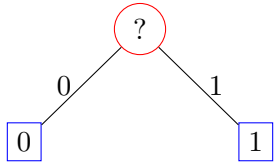
\begin{proof}
From our finite tree in Figure~\ref{fig:plus1} we get that the $2$-adic valuations of our sequence $x^2+1$ can only ever be
$$
\nu_{2}(n^2 + 1) = \left\{\begin{array}{cc}
0, & n\equiv 0 \mod 2  \\
1, & n \equiv 1 \mod 2. 
\end{array}
\right.
$$
Then we translate our valuation equations to get $x^2+1=2^0y$ and $x^2+1=2^1y$, which we know must have solutions since our tree stated that there were $x$ values that we could plug in to get exact valuation $0$ or $1$. 

Now we need to solve for $y$, which gives us
\[y=x^2+1,\]
where $x$ is even and
\[y=\frac{x^2+1}{2},\]
where $x$ is odd. 

There are no other non-trivial $c$ values that will be solutions to this Diophantine equation because there are no valuations greater than $1$ in our valuation tree. 

Since the tree is finite we know that any other classifications of our $x$ value modulo $2^c$ will only result in one of the two stated valuations. 
\end{proof}

\begin{thm}
The equation $x^2 + 3 = 2^c y$ has no solutions when $c \neq 0, 2$.  If $c=0$, $y$ is odd whenever $x$ is even.  If $c = 2$, $y$ is even whenever $x$ is odd.  
\end{thm}

\begin{figure}[H]
\begin{tikzpicture}[sibling distance=8em,
]]
\tikzset{circle node/.style={shape=circle,draw=red,rounded corners,
    draw, align=center,
    top color=white, bottom color=white},square node/.style={shape=rectangle,draw=blue,
    draw, align=center,
    top color=white, bottom color=white}}
  \node[circle node] {$?$}
     child {node[square node] {$1$}edge from parent node[left]{$0$} }
     child { node[square node] {$2$} 
          edge from parent node[right]{$1$}
        }
        ;
\end{tikzpicture}
\caption{The finite valuation tree for $x^2+3$.}
\label{fig:plus3}
\end{figure}
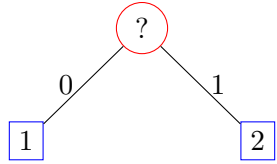

\begin{proof}
We summarize the finite tree in Figure~\ref{fig:plus3} here:
$$
\nu_{2}(n^2 + 3) = \left\{\begin{array}{cc}
0, & n\equiv 0 \mod 2  \\
2, & n \equiv 1 \mod 2. 
\end{array}
\right.
$$
The $2$-adic valuation translates to the following equation:
\[x^2+3=2^cy.\]
Our finite valuation tree says that the only possible values of $c$ (where $y$ is odd) will be $0$ and $2$.
So we have equations
\begin{align}
x^2+3 &= y \label{eqn:plus3even}\\
x^2+3 &= 2^2y.\label{eqn:plus3odd}
\end{align}

Further the valuation tree says we only have solutions for $c=0$ if $x$ is even. So let $x=2$ and find $y$ in \ref{eqn:plus3even}. So one solution for our Diophantine equation is $(2,7,0).$ We can find all solutions for the Diophantine equation similarly: $(2n,(2n)^2+3,0),$ for $n\in\ZZ.$

Now for equation ~\ref{eqn:plus3odd}, the tree says this valuation only occurs when $x$ is odd. So our solutions are $(2n+1,\frac{(2n+1)^2+3}{4},2)$, for $n\in\ZZ.$
\end{proof}

\begin{thm}
The equation $x^2 + 4 = 2^c y$ has no solutions when $c \neq 0, 2, 3$.  
\end{thm}

 \begin{figure}[H]
\begin{tikzpicture}[sibling distance=8em,
]]
\tikzstyle{level 3}=[level distance=15mm,sibling distance=1.5cm]
\tikzstyle{level 4}=[level distance=15mm,sibling distance=1.2cm]
\tikzset{circle node/.style={shape=circle,draw=red,rounded corners,
    draw, align=center,
    top color=white, bottom color=white},square node/.style={shape=rectangle,draw=blue,
    draw, align=center,
    top color=white, bottom color=white}}
  \node[circle node] {$?$}
     child {node[circle node] {$2$}
        child { node[square node]  {2}  edge from parent node[left]{$0$} }
        child { node[square node]  {3} edge from parent node[right]{$2$} }
     edge from parent node[left]{$0$} 
     }
     child { node[square node] {$0$}
      edge from parent node[right]{$1$} 
    } 
        ;
\end{tikzpicture}
\caption{The finite valuation tree for $x^2+4.$}
\label{fig:plus4}
\end{figure}
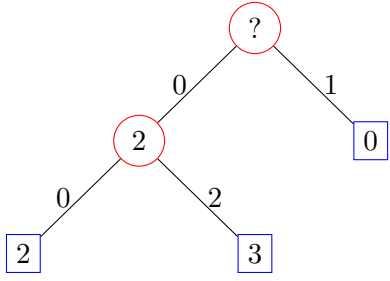

\begin{proof}
We summarize our valuation tree in Figure~\ref{fig:plus4} here:
$$
\nu_{2}(n^2 + 4) = \left\{\begin{array}{cc}
0, & n\equiv 1 \mod 2  \\
2, & n \equiv 0 \mod 4\\
3, & n \equiv 2 \mod 4.
\end{array}
\right.
$$
Again, the $2$-adic valuation translates to the following equation:
\[x^2+3=2^cy.\]
Our finite valuation tree says that the only possible values of $c$ (where $y$ is odd) will be $0,\ 2,$ and $3$.
So we have equations
\begin{align}
x^2+4 &= y \label{eqn:plus4odd}\\
x^2+4 &= 2^2y.\label{eqn:plus4divby4}\\
x^2+4 &= 2^3y.\label{eqn:plus4notdivby4}
\end{align}

Further the valuation tree says we only have solutions for $c=0$ if $x$ is odd. So let $x=1$ and find $y$ in \ref{eqn:plus4odd}. So one solution for our Diophantine equation is $(1,5,0).$ We can find all solutions for the Diophantine equation similarly: $(2n+1,(2n+1)^2+4,0),$ for $n\in\ZZ.$

Now for equation ~\ref{eqn:plus4divby4}, the tree says this valuation only occurs when $x$ is divisible by 4. So our solutions are $(2^2n,\frac{(2^2n)^2+4}{4},2)$, for $n\in\ZZ.$

Finally, for equation ~\ref{eqn:plus4notdivby4}, the tree says this valuation only occurs when $x$ is $2\mod 4$. So our solutions are $(2^2n+2,\frac{(2^2n+2)^2+4}{8},2)$, for $n\in\ZZ.$

\end{proof}

Now we will work out an example for our cubic Diophnatine equation where $D=8^k(2j+1)=1$. So we are working with the Diophantine equation $x^3 + 1 = 2^c y$.  The valuation tree is shown in Figure~\ref{fig:cubeplus1}.

 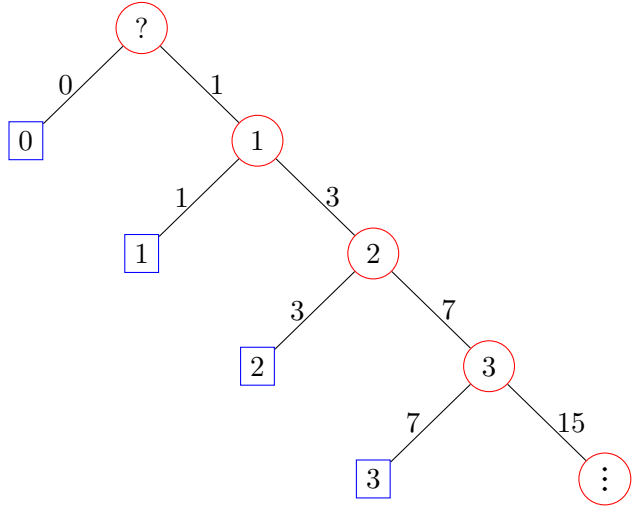
\begin{figure}[H]
\begin{tikzpicture}[sibling distance=8em,
]]
\tikzset{circle node/.style={shape=circle,draw=red,rounded corners,
    draw, align=center,
    top color=white, bottom color=white},square node/.style={shape=rectangle,draw=blue,
    draw, align=center,
    top color=white, bottom color=white}}
  \node[circle node] {$?$}
     child {node[square node] {0}edge from parent node[left]{$0$} }
     child { node[circle node] {$1$} 
        child {node[square node] {$1$}
        edge from parent node[left]{$1$}
        }
        child {node[circle node] {$2$}
            child {node[square node] {$2$}
            edge from parent node[left]{$3$} }
            child {node[circle node] {$3$}
                child {node[square node] {$3$} edge from parent node[left]{$7$}  }
                child {node[circle node] {$\vdots$}
                     edge from parent node[right]{$15$}  }
        edge from parent node[right]{$7$}     
            }  edge from parent node[right]{$3$}  
        }edge from parent node[right]{$1$}
        }
        ;
\end{tikzpicture}
\caption{The infinite valuation tree for $x^3+1$ where the tree continues to split on the right most node indefinitely.}
\label{fig:cubeplus1}
\end{figure}

We'll see from the theorem below and the tree that the 2-adic valuation tree for $\nu_{2}(n^3 + 1)$ has range of $\ZZ_{\geq 0}$. Also at each level $c$, there are two branches yielding one terminating node with value $c-1$ and one non-terminating node with minimum value $c$. 
\begin{thm}
The solutions of the cubic Diophantine equation $x^3 + 1 = 2^c y,$ $y$ odd, follow from
\begin{equation}\label{range cubic D=1}
\nu_{2}(n^3 + 1) =\left\{
\begin{array}{cc}
0, & n \equiv 0 \mod 2^1\\
1, & n \equiv 1 \mod 2^2\\
2, & n \equiv 3 \mod 2^3\\
3, & n \equiv 7 \mod 2^4\\
4, & n \equiv 15 \mod 2^5\\
\vdots\\
c, & n \equiv 2^{c}-1 \mod 2^{c+1}
\end{array}
\right.
\end{equation}
So there are solutions for all $c\in\ZZ_{\geq 0}$ with corresponding $x\equiv 2^c-1\mod 2^{c+1}$
\end{thm}

\begin{proof}
First we show that the tree is infinite and will have the valuations as stated. Fix $c \geq 0$. Suppose $n = 2^c - 1 \mod 2^{c+1}.$  Then
$$
\begin{array}{lll}
(2^c - 1)^3 + 1 &=& 2^{3c} - 3\cdot 2^{2c} +  3\cdot 2^c - 1 + 1\\ 
&=& 2^{c}(2^{2c} - 3\cdot 2^{c} +  3).\\
\end{array}
$$
Observe that $2^{2c} - 3\cdot 2^{c} +  3$ is odd for any $c$.

Hence we know that solutions will be of the form $(x,y,c)$ where $c\in \ZZ_{\geq 0}.$ For each $c$, we have a corresponding $x$ value that we know is of the form $x\equiv 2^c-1\mod 2^{c+1}$ because the right most branch in the tree is the one that will always continue. 
\end{proof}

\section{Conclusion}\label{sec:conclusion}
What is interesting about this approach is that it gives us a handle on Diophantine equations of the form $x^\ell+D=p^cy,$ for any prime $p$. Through studying classifications of $x$ and creating our valuation trees we're able to determine, for which $c$, $x^\ell +D=p^cy$ has non-trivial integer solutions. 

For example, if we were to have studied $x^2+7=2^cy$ with a traditional tool such as the one in Bilu, Hanrot, and Voutier in \cite{bhvlucas} we would have discovered that there are finitely many solutions for $c=1,2$ and $3$, but are unable to determine for $c>3.$ 

So one direction we want to head is to understand the $p$-adic valuation trees of $x^2+D$ or $x^3+D$. It would also be interesting to understand other polynomials. Then we can use these results to study the related Diophantine equations.

\subsection*{Acknowledgements} We would like to thank REUF, as the original research questions began at the one week workshop. We'd like to thank Justin Trulen and Jane Long for all the helpful discussions and recommendations on the citations. We want to thank Diego Villamizar Rubiano for his input on the paper in its final stages. We especially want to thank Zoom for working so well and having a good ``share screen" feature. Sadly, no coffee shops were visited in the making of this article. Bianca Thompson was partially supported by the Gore Summer Research Grant through part of this project.

\appendix
\section{Tables of example solutions}\label{sec:app}

\begin{table}[h]
\[\begin{array}{l|l|l}
c & x_{c} & y_{c} \\
\hline
3 & 1 & 1\\
4 & 3 & 1\\
5 & 5 & 1\\
6 & 11 & 2\\
7 & 11 & 1\\
8 & 53 & 11\\
9 & 75 & 11\\
10 & 181 & 32 \\
11 & 181 & 16\\
12 & 181 & 8\\
13 & 181 & 4\\
14 & 181 & 2\\
15 & 181 & 1\\
16 & 16203 & 4006\\
17 & 16203 & 2003\\
18 & 49333 & 9284\\
19 & 49333 & 4642\\
20 & 49333 & 2321\\
21 & 474955 & 107566\\
22 & 474955 & 53783\\
23 & 1622197 & 313702\\
24 & 1622197 & 156851\\
25 & 6766411 & 1364479\\
26 & 10010805 & 1493338\\
\hline
\end{array}\]
\caption{Some example solutions to our Diophantine equation $x^2+7=2^cy$, $y$ odd, using our recursion.}
\label{tab:recursion}
\end{table}

Tables~\ref{table:deg2} and~\ref{table:deg3} explore solutions for $x^2 + D = 2^c y$ and $x^3+D=2^cy$. We showcase some of the solutions $(x,y,c)$ if we're given values of $D$.

\begin{table}[H]
    \centering
    \begin{tabular}{l|l|l}
    $D$ & $c$ &$(x,y,c)$\\
    \hline
    8 & 0, 2, 3 & $(1,9,0),\ (2,3,2),\ (4,3,3)$\\
    9 & 0, 1 & $(2,13,0),\ (1,5,1)$\\
    10 & 0, 1 & $(1,11,0),\ (2,7,1)$\\
    11 & 0, 2 & $(2,15,0),\ (1,3,2)$\\
    12 & 0, 2, 4 & $(1,13,0),\ (4,7,2),\ (2,1,4)$\\
    13 & 0, 1 & $(2,17,0),\ (1,7,1)$\\
    14 & 0, 1 & $(1,15,0),\ (2,9,1)$\\
    16 & 0, 2, 4, 5 & $(1,17,0),\ (2,5,2),\ (8,5,4),\ (4,1,5)$\\
    17 & 0, 1 & $(2,21,0),\ (1,9,1)$\\
    18 & 0, 1 & $(1,19,0),\ (2,11,1)$\\
    19 & 0, 2 & $(2,23,0),\ (1,5,2)$\\
    20 & 0, 2, 3 & $(1,21,0),\ (4,9,2),\ (2,3,3)$\\
    21 & 0, 1 & $(2,25,0),\ (1,11,1)$\\
    22 & 0, 1& $(1,23,0),\ (2,13,1)$\\
    24 & 0, 2, 3& $(1,25,0),\ (2,7,2),\ (4,5,3)$\\
    25 & 0, 1& $(2,29,0),\ (1,13,1)$\\
    26 & 0, 1& $(1,27,0),\ (2,15,1)$\\
    27 & 0, 2& $(2,31,0),\ (1,7,2)$\\
    29 & 0, 1& $(2,33,0),\ (1,15,1)$\\
    30 & 0, 1& $(1,31,0),\ (2,17,1)$\\
    32 & 0, 2, 4, 5& $(1,33,0),\ (2,9,2),\ (4,3,4),\ (8,3,5)$\\
    33 & 0, 1& $(2,37,0),\ (1,17,1)$\\
    34 & 0, 1& $(1,35,0),\ (2,19,1)$\\
    35 & 0, 2& $(2,39,0),\ (1,9,2)$\\
    36 & 0, 2, 3& $(1,37,0),\ (4,13,2),\ (2,5,3)$\\
    37 & 0, 1& $(2,41,0),\ (1,19,1)$\\
    38 & 0, 1& $(1,39,0),\ (2,21,1)$\\
    40 & 0, 2, 3&$(1,41,0),\ (2,11,2),\ (4,7,3)$
    \end{tabular}
    \caption{Examples of some solutions for $x^2+D=2^cy$ for specific $D$ values.}
    \label{table:deg2}
\end{table}

\begin{table}[H]
    \centering
    \begin{tabular}{l|l|l}
    $D$ & $c$ & $(x,y,c)$\\
    \hline
    2 & 0, 1 & $(1,3,0),\ (2,5,1)$\\
    4 & 0, 2 & $(1,5,0),\ (2,3,2)$\\
    6 & 0, 1 &$(1,7,0),\ (2,7,1)$\\
    10 & 0, 1 & $(1,11,0),\ (2,9,1)$\\
    12 & 0, 2 & $(1,13,0),\ (2,5,2)$\\
    14 & 0, 1 &$(1,15,0),\ (2,11,1)$\\
    16 & 0, 3, 4 & $(1,17,0),\ (2,3,3),\ (4,5,4)$\\
    18 & 0, 1 & $(1,19,0),\ (2,13,1)$\\
    20 & 0, 2 &$(1,21,0),\ (2,7,2)$\\
    22 & 0, 1 & $(1,23,0),\ (2,15,1)$\\
    26 & 0, 1 & $(1,27,0),\ (2,17,1)$\\
    28 & 0, 2 & $(1,29,0),\ (2,9,2)$\\
    30 & 0, 1 & $(1,31,0),\ (2,19,1)$\\
    32 & 0, 3, 5 &$(1,33,0),\ (2,5,3),\ (4,3,5)$\\
    34 & 0, 1 & $(1,35,0),\ (2,21,1)$\\
    36 & 0, 2 & $(1,37,0),\ (2,11,2)$\\
    38 & 0, 1 & $(1,39,0),\ (2,23,1)$\\
    42 & 0, 1 & $(1,43,0),\ (2,25,1)$\\
    44 & 0, 2 & $(1,45,0),\ (2,13,2)$\\
    46 & 0, 1 & $(1,47,0),\ (2,27,1)$
    \end{tabular}
    \caption{Examples of some solutions for $x^3+D=2^cy$ for specific $D$ values.}
    \label{table:deg3}
\end{table}

\bibliographystyle{plain}
\bibliography{ref}

\end{document}